\documentclass[12pt,a4paper]{amsart}
\usepackage{amsmath, amssymb, amscd,amsthm, amsfonts}
\usepackage[mathscr]{eucal}
\usepackage[all]{xy}
\newtheorem{theorem}{Theorem}[section]

\theoremstyle{definition}
\newtheorem{definition}[theorem]{Definition}

\newtheorem{proposition}[theorem]{Proposition}
\theoremstyle{remark}
\newtheorem{remark}[theorem]{Remark}
\numberwithin{equation}{section} 
\theoremstyle{plain}
\newtheorem{acknowledgement}{Acknowledgement}

\newtheorem{corollary}[theorem]{Corollary}

\def\S{\mathbb S}
\def\C{\mathbb C}

\def\H{\mathbb H}
\def\P{\mathbb P}

\def\G{{\mathbb G}L(2, \H)}
\def\V{\mathbb V}

\def\T{\mathbb T}
\def\m{\mathcal M}
\def\E{\mathbb E}

\def\H{\mathbb H}
\def\h{\bf H}
\def\o{\ddot{\hbox{o}}}
\def\g{{\mathbb G}L}
\def\R{\mathbb R}

\def\V{\mathbb V}

\def\T{\mathbb T}
\def\P{\mathbb P}
\def\K{\mathbb K}
\def\S{\mathbb S}
\def \s{\mathcal S}

\newcommand{\secref}[1]{section~\ref{#1}}
\newcommand{\thmref}[1]{Theorem~\ref{#1}}

\newcommand{\eqnref}[1]{~{\textrm(\ref{#1})}}

\begin{document}


\title[Algebraic characterization of the isometries of the hyperbolic 5-space]
{Algebraic characterization of the isometries of the hyperbolic 5-space}

\author[Krishnendu Gongopadhyay]{Krishnendu Gongopadhyay}

\address{School of Mathematics, Tata Institute of Fundamental Research, Colaba, Mumbai 400005, India}

\email{krishnendug@gmail.com}

%\date{May 7, 2009}

%\footnote{{\it Mathematics Subject Classification(2000). \hspace{.1in}}{Primary 51M10; Secondary 51F25}}

\begin{abstract}
Let $\G$ be the group of invertible $2 \times 2$ matrices over the division algebra $\H$ of quaternions. $\G$ acts on the hyperbolic $5$-space as the group of orientation-preserving isometries. Using this action we give an algebraic characterization of the dynamical types of the orientation-preserving isometries of the hyperbolic $5$-space. Along the way we also determine the conjugacy classes and
 the conjugacy classes of centralizers or the $z$-classes in $\G$.
\end{abstract}
\maketitle

\section{Introduction}\label{intro}
Let $\h^{n+1}$ denote the $(n+1)$-dimensional hyperbolic space. The conformal boundary of the hyperbolic space is the $n$-dimensional sphere $\S^n$. Let $\E^n$ denote the $n$-dimensional euclidean space. We identify $\S^n$ with the extended euclidean space $\hat \E^n=\E^n \cup \{\infty\}$. Let $I_o(n+1)$ denote 
the group of orientation-preserving isometries of $\h^{n+1}$.   
 Classically, one uses the ball model of $\h^{n+1}$ to define the dynamical type of an isometry. In this model an
isometry is elliptic if it has a fixed point on the disk. An isometry is
 parabolic, resp. hyperbolic, if it is not elliptic and has one, resp. two 
fixed points on the conformal boundary of the hyperbolic space. If in addition to the fixed points one also consider the ``rotation-angles'' of an isometry, the above classification of the dynamical types can be made finer.

Let $S$ be an orthogonal transformation of $\E^n$. The rotation angles of $S$ correspond to each pair of complex conjugate eigenvalues of $S$. For each pair of complex conjugate eigenvalues $\{e^{i \theta}, e^{-i \theta}\}$, $-\pi \leq \theta \leq \pi$, $\theta \neq 0$, we assign a rotation angle to $S$.
 If $S$ has $k$ rotation angles, it is called a $k$-rotation. Now suppose $T$ is an isometry of $\h^{n+1}$. Then it follows from the description of the conjugacy class of $T$ that one can associate to $T$ an orthogonal transformation $A_T$ of $\E^n$. For our purpose it is enough to consider the case when $n$ is even, that is the dimension of $\h^{n+1}$ is odd. In this case, by Lefschetz fixed-point theorem,  every isometry has a fixed point on $\S^n$. Hence the restriction of $T$ to $\S^n$ can be conjugated to a similarity $f_T$ of $\E^n$. 
 The orthogonal transformation $A_T$ is associated to this similarity $f_T$ of $\E^n$. With respect to a suitable coordinate system, $f_T$ is of the form $A_T x +b$. 

 We call $T$ a $k$-rotatory elliptic (resp. $k$-rotatory parabolic, resp $k$-rotatory hyperbolic) if it is elliptic (resp. parabolic, resp. hyperbolic) and $A_T$ is a $k$-rotation. A $0$-rotatory parabolic (resp. a $0$-rotatory hyperbolic) is called a translation (resp. a stretch). For more details of this classification cf. \cite{kg}.

Recall that in dimension $3$, the group $\mathbb SL(2, \C)$, or equivalently, $\mathbb GL(2, \C)$ acts as the linear fractional transformations of the boundary sphere $\S^2$ and the dynamical types are characterized by the trace \cite{beardon}, \cite{rat} and $\frac{{trace^2}}{\hbox{det}}$ \cite{kg} respectively. In \cite{parker} Parker {\it et al}  
have given an algebraic characterization of the dynamical types of the
orientation-preserving isometries of the hyperbolic $4$-space. Parker {\it et al} have offered the characterization after identifying the group of isometries with a proper subgroup of $\G$ which preserves the unit disk on $\H$.

Our interest in this paper are the orientation-preserving isometries of the hyperbolic $5$-space. The conformal boundary $\S^4$ of the hyperbolic 5-space is identified with the extended quaternionic plane $\hat \H=\H \cup \{\infty\}$.  The group $\G$ acts on $\hat \H$ as the linear fractional transformations:
$$\begin{pmatrix} a & b \\ c & d \end{pmatrix} : Z \mapsto (aZ +b)(cZ +d)^{-1}.$$
Under this action  $\G$ can be identified with the identity component of the full group of M$\o$bius transformations of $\S^4$. We have proved this fact in \secref{linfrac}. Another proof can be found in \cite{wilker}. Comparable versions are available in \cite{ahlfors}, \cite{cw}, \cite{kellerhals}, \cite{kellerhals2}, \cite{waterman}. The proof we give here is different from the existing proofs and is more geometric.  

The group $\G$ can be embedded in $\mathbb GL(4, \C)$ as a subgroup. Using this embedding and the representation of the isometries of $\h^5$ as $2 \times 2$ matrices over the quaternions, we offer an algebraic characterization of the dynamical types. Our main theorem is the following.
\begin{theorem}\label{mainth}
Let $f$ be an orientation-preserving isometry of $\h^5$. Let $f$ be induced by $A$ in $\G$. Let $A_{\C}$ be the corresponding element in $\mathbb G L(4, \C)$. Let the characteristic polynomial
  of $A_{\C}$ be 
$$\chi(A_{\C})=x^4-2a_3x^3+a_2x^2-2a_1x+a_0.$$
Then $a_0>0$. 
Define, 
$$c_1={\frac{a_1^2}{a_0 \sqrt a_0}}\:, \; c_2={\frac{a_2}{ \sqrt a_0}}\:, \;
c_3={\frac{a_3^2}{ \sqrt a_0}}\:.$$ 
Then we have the following.

$(i)$ $A$ acts as an $2$-rotatory hyperbolic if and
only  if $c_1\neq c_3$.

$(ii)$ $A$ acts as a $2$-rotatory elliptic if and only if 
$$c_1=c_3, \;c_2 < c_1+2.$$

$(iii)$ $A$ acts as 
an $1$-rotatory hyperbolic if and only
if 
$$c_1=c_3, \;c_2 > c_1+2.$$ 

 $(iv)$ $A$ acts as a translation if and only if 
$$c_1=c_3,\; c_2= c_1+2, \; c_1=4, $$
and $A$ is not a real diagonal matrix.

$(v)$ $A$ acts as a stretch if and only if 

\centerline{$c_1=c_3$, $c_2= c_1+2$, and 
$ c_1 > 4$. }

$(vi)$ $A$ acts as an  $1$-rotatory elliptic or an
$1$-rotatory parabolic if and only if 

\centerline{$c_1=c_3$, $c_2= c_1+2$,  and $c_1 < 4$.} 

Moreover if the characteristic polynomial of $A_{\C}$ is equal to its minimal polynomial, then $A$ acts as an $1$-rotatory parabolic. Otherwise, it acts as an $1$-rotatory elliptic. 
\end{theorem}
The theorem is proved in \secref{pmainth}. The action of $\G$ on $\hat \H$ also enables us to determine the conjugacy classes in $\G$. The conjugacy classes are classified in \secref{cc}. In \secref{z-class} we determine the centralizers of elements and the $z$-classes in $\G$.  We recall that two elements $x$ and $y$ in a group $G$ are said to be in the same $z$-class if the corresponding centralizers $Z(x)$ and $Z(y)$ are conjugate in $G$. From the description of the conjugacy classes we see that there are infinitely many of them in $\G$. But it turns out that the number of $z$-classes is finite. This gives another partition of the isometries into finitely many classes.  We compute the precise number of $z$-classes.
\begin{theorem}\label{zclass}
There are exactly seven $z$-classes in $\G$. 
\end{theorem}
Thus $\G$ provides a special example to the philosophy that was suggested in \cite{rkrjm} and was elaborated for linear and affine maps in \cite{rskajm}, and for isometries of pseudo-riemannian geometries of constant curvature in \cite{thesis}. 

Much after the announcement \cite{kg1} of the main result of this paper, I have come to know about the work of Cao \cite{cao}, Foreman \cite{foreman}, and, Parker and Short \cite{ps}. These authors also have given some algebraic characterizations of the dynamical types. In all these papers the authors have got the characterization by considering the representation of the isometries in $\S L(2, \H)$, which is the group $\G$ with some ``normalization". In our work we have given the algebraic characterization without normalizing the matrices, and hence it is different from the ones obtained by these authors.

\medskip {\bf Notations.} Let $\K$ be a division ring. Then  $\K^{\ast}$ is the multiplicative group of $\K$, i.e. $\K^{\ast}=\K-\{0\}$. For a non-zero $x$ in $\K$, $Z(x)$ denotes the centralizer of $x$ in $\K^{\ast}$. The group $\mathbb G L (m, \K)$ is the group of all invertible $m \times m$ matrices over $\K$. The group $\T(m, \K)$ is a subgroup of $\mathbb GL(m, \K)$ and consists of all upper-triangular matrices with $1$ on the diagonal. For a group $G$, by $PG$ we denote the group $G/Z(G)$, where $Z(G)$ is the center of $G$. 

\section {Preliminaries}\label{prel}
\subsection{ Dynamical types in hyperbolic and Euclidean geometry}\label{he}
Let $O(n)$ denote the orthogonal group of $\E^n$, i.e. $O(n)$ is the group of isometries of the positive-definite quadratic form on $\E^n$. To each similarity $f$ of $\E^n$, there is an associated  semisimple transformation which belongs to the compact group $O(n)$. In fact after choosing a coordinate system, one can write $f$ as
$$f(x)=r Ax +b,\;\;r>0,\;A \in O(n),\; b \in \E^n.$$
Note that $f$ is not an isometry of $\E^n$ if and only if $r \neq 1$. 

For each pair of complex conjugate eigenvalues $\{e^{i \theta}, e^{-i \theta}\}$, $0 < \theta \leq \pi$, of $A$, we assign a rotation angle to $f$.  
 An orientation-preserving similarity of  $\E^{n}$ is called a  {\it  $k$-rotatory elliptic (resp. parabolic)} 
if it is an isometry of $\E^n$ and has $k$-rotation angles and a fixed point (resp. no fixed point) on $\E^n$. An orientation-preserving
 similarity is called a {\emph{ $k$-rotatory hyperbolic}} if it has $k$-rotation angles and is not an isometry of $\E^n$. 

 Let $\mathcal S^+(n)$ denote the group of all
orientation-preserving similarities of
$\E^n$. Consider the ball model of $\h^{n+1}$, and identify its conformal boundary with $\hat \E^n=\E^n \cup \{\infty\}$. When $n=2m$, by Lefschetz fixed point theorem we see that every isometry of $\h^{2m+1}$ has a fixed point on the conformal boundary. Up to conjugation we can choose the fixed point to be $\infty$. It follows that the restriction of any element $f$ of $I_o(2m+1)$ to the conformal boundary is conjugate to an  element $s_f$ of $\mathcal S^+(2m)$. 
Also, by Poincar\'e extension every element $g$ of $\mathcal S^+(2m)$ can be extended uniquely to an 
isometry $\tilde g$ of $\h^{2m+1}$ . 
 Further, $f$ and $s_f$ have the same number of non-zero rotation-angles. Hence corresponding to each conjugacy class 
of $I_o(2m+1)$ we have a conjugacy class in $\mathcal S^+(2m)$ 
so that they are represented
by the same element and vice versa. This gives us the following.

\smallskip {\it The dynamical types of orientation-preserving 
isometries of   $\h^{2m+1}$ are in 
  bijective type-preserving correspondence with the dynamical
  types of orientation-preserving similarities of  $\E^{2m}$.} 

This allows us to identify the dynamical types of orientation-preserving isometries of $\h^5$ and the dynamical types of orientation-preserving similarities of $\E^4$.

\subsection { Review of Quaternions}
The space of all quaternions $\H$ is the four dimensional real algebra with
basis $\{1,\;i,\;j,\;k\}$ and multiplication rules
$i^2=j^2=k^2=-1$, $ij=-ji=k,\;jk=-kj=i,\;ki=-ik=j$. 
For a quaternion $x=x_0+x_1i+x_2j+x_3k$ we define
$\Re x=x_0$ and $\Im x=x_1i+x_2j+x_3k$. The {\it norm} of $x$ is
defined as $|x|=\sqrt{x_0^2+x_1^2+x_2^2+x_3^2}$. The conjugate of $x$ is defined by
$\bar x= x_0-x_1i-x_2j-x_3k$

We choose $\C$ to be the subspace of $\H$ spanned by $\{1, i\}$. 
With respect to this choice of $\C$ we can write,  $\H=\C \oplus \C j$.
That is, every element $a$ in $\H$ can be uniquely expressed
as $a=c_0+c_1j$, where $c_0$, $c_1$ are complex numbers. Similarly we can also write $\H=\C \oplus j \C$. 
\begin{definition} Two quaternions $a$ and $b$ are {\emph{similar}}
 if there exists a non-zero quaternion $v$ such that $a=vbv^{-1}$.
 \end{definition}
\begin{proposition} \cite{brenner} Two quaternions are similar if and only if
$\Re a=\Re b$ and $|a|=|b|$.
\end{proposition}

\begin{corollary}
The similarity class of every quaternion $\alpha$ contains a pair of
  complex conjugates with absolute-value $|\alpha|$ and real part equal to
  $\Re \alpha$.
  \end{corollary}

\begin{proposition}
The group $\G$ can be embedded in the group $\g (4, \C)$.
\end{proposition}
\begin{proof} 
Let $A=\begin{pmatrix} a & b \\ c&
    d \end{pmatrix}$ be an element in $\G$. Write $a=a_0+ja_1$,
    $b=b_0+jb_1$, $c=c_0+jc_1$, $d=d_0+jd_1$. Then we can write 
$A=A_0+jA_1$, where
$$A_0=\begin{pmatrix}a_0 & b_0 \\ c_0 & d_0 \end{pmatrix},\;
A_1=\begin{pmatrix}a_1 & b_1 \\ c_1 & d_1 \end{pmatrix},$$
Now we define the map $p:M(2, \H)\to M(4, \C)$ as
$$p: A=A_0+jA_1 \mapsto \begin{pmatrix} A_0 & -\bar A_1 \\ A_1 & \bar
    A_0 \end{pmatrix}=A_{\C}.$$
The left action of $A$ on $(z_0 +jz_1, w_0 +jw_1)^t$ induces linear left action of $A_{\C}$ on $(z_0, w_0, z_1, w_1)^t$. 
It is easy to verify that $p$ is an injective homomorphism and
maps $\G$ into $\g (4, \C)$.
\end{proof}
\begin{remark}\label{rk1}
For more details of the proof of the above proposition cf. \cite{ask} section-3, \cite{lee} section-2. It is necessary to put $j$ on the left. In order to see this observe

\medskip $
\begin{pmatrix}a_0 + a_1j & b_0+ b_1j \\c_0 + c_1j & d_0 + d_1j\end{pmatrix} \begin{pmatrix} z_0 + z_1 j \\ w_0 + w_1 j \end{pmatrix}$

$=\begin{pmatrix} 
(a_0z_0 -a_1 \bar z_1 + b_0w_0 -b_1\bar w_1)+ (a_0 z_1 -a_1 \bar z_0+b_0 w_1-b_1 \bar w_0)j\\(c_0z_0 -c_1 \bar z_1 + d_0w_0 -d_1\bar w_1)+ (c_0 z_1 -c_1 \bar z_0+d_0 w_1-d_1 \bar w_0)j
     \end{pmatrix}$

\medskip \noindent and so $A$ does not act on the left complex linearly on $(z_0, w_0, z_1, w_1)^t$. 
\end{remark}

\medskip We call $\V$ a right vector space over $\H$, if $\V$ is an additive group and for $v$ in $\V$, $\lambda$ in $\H$, the scalar multiplication $v.\lambda$ is defined on the right.  Now consider $\H^2$ as a right vector-space over $\H$. 
Let $\P^1(\H)$ be the projective space over the quaternions. 
That is, $\P^1(\H)$ is obtained by identifying one-dimensional right subspaces in $\H^2$.  
Let $\begin{bmatrix}x \\
y \end{bmatrix}$ denotes the equivalence class of ${\begin{pmatrix} x \\
y \end{pmatrix}} \in \H^2-\{0\}$. We define $\infty$ to be the equivalence
class $\begin{bmatrix} x \\
0 \end{bmatrix}$.
We can identify $\P^1(\H)$ with the extended
  quaternionic space $\hat \H=\H\cup \infty$ by the map:
$${\begin{bmatrix} u \\
v \end{bmatrix}} \mapsto 
\left \{\begin{array}{ll} uv^{-1} \hbox { if $v \neq 0$}\\
\infty \hbox { otherwise}
\end{array} \right. $$
Topologically $\P^1(\H)$ is homeomorphic to the four sphere $\S^4$. 

The group $\G$ acts on $\H^2$ on the
left by the following action:
$$ A={\begin{pmatrix}a & b \\
c & d \end{pmatrix}}: {\begin{pmatrix} u \\
v\end{pmatrix}} \mapsto {\begin{pmatrix} au + bv\\
cu+dv\end{pmatrix}}.$$
This induces the action of $\G$ on $\P^1(\H)$ as the group
of quaternionic linear-fractional transformations:
$${\begin{pmatrix}a & b \\
c & d \end{pmatrix}}Z=(aZ+b)(cZ+d)^{-1},$$
where $Z=uv^{-1}$ corresponds to the element 
$\begin{bmatrix} u \\
v \end{bmatrix}$. From now on we identify $\P^1(\H)$ and $\hat \H$ with the boundary sphere $\S^4$ of $\h^5$, where $\H$ is identified with $\R^4$. 

The above action enables us to give a geometric proof of the following proposition regarding eigenvalues of elements in $\G$. 
\begin{definition} Let $A$ be an element in $\G$. An element $\lambda$ of $\H$ is 
said to be a {\emph{right eigenvalue}} of $A$ if $Av=v \lambda$ for some
non-zero $v$ in $\H^2$.
\end{definition}
\begin{proposition}
Every matrix in $\G$ has a right eigenvalue.
\end{proposition}
\begin{proof}Let $A$ be an element in $\G$. Let $\hat A$ be the linear fractional transformation of $\P^1(\H)$ induced by $A$. By Lefschetz fixed-point theorem, $\hat A$ has a fixed point on the projective space. 
The fixed points of $\hat A$  
correspond to $A$-invariant one dimensional right subspaces in $\H^2$. This completes the proof.
\end{proof}
The referee kindly pointed out that an easier proof of the above result can be found in \cite{lee} (Theorem-2). 
\section{ M$\o$bius transformations of $\S^4$}\label{linfrac}
Let $\m(4)$ denote the group of M$\o$bius transformations of $\S^4$, that is, $\m(4)$ is the group of diffeomorphisms of $\S^4$ which is generated by inversions in $3$-spheres.
Then we have the following
\begin{proposition}
$\m(4)=\m^+(4) \cup \m^{-}(4)$
  where
$$\m^+(4)=\bigg \{Z \mapsto (aZ+b)(cZ+d)^{-1}\;:\; {\begin{pmatrix} a & b \\
c & d \end{pmatrix}}\in \G \bigg \},$$
$$\m^-(4)=\bigg \{Z \mapsto (a \bar Z+b)(c \bar Z+d)^{-1}\;:\; 
{\begin{pmatrix} a & b \\ c & d \end{pmatrix}}\in \G \bigg \}.$$ 
\end{proposition}

\begin{proof}
Let $G=\m^+(4) \cup \m^-(4)$.  In quaternionic expression, the inversion in
a sphere   $S_{a,r}= \{Z\;:\; |Z-a|=r \}$ is given by
$${\sigma_{{a,r}}}: Z \mapsto a+ r^2 (\bar Z-\bar a)^{-1}.$$ 
$\hbox{ Since, }\;\sigma_{a,r}(Z)=a+r^2(\bar Z -\bar a)=(a\bar Z -a \bar a +r^2)(\bar Z-\bar a)^{-1},\;$
hence $\sigma_{a,r}$ is an element in $G$. 
The reflections in hyperplanes  of $\H$ can be written in the form 
$$r_{\lambda,a}: Z \mapsto a-\lambda (\bar Z-\bar a)  \lambda,  \;|\lambda|=1.$$  
Since, $r_{\lambda, a}(Z)=(-\lambda \bar Z+\lambda \bar a+ a \bar \lambda)(\bar \lambda)^{-1}$, hence $r_{\lambda,a}$ is an element in $G$.  
 Thus every
inversion in a $3$-sphere of $\S^4$ is contained in the group $G$. Hence $\m(4) \subset G$.

 Now we shall show the reverse inclusion. The Euclidean similarities form a subgroup $\s(4)$ of $G$. In quaternionic expression, 
$$ \s^+(4)=\{Z \mapsto r \lambda Z \mu+b\;|\;\lambda, \mu, b \in
\H,\;|\lambda|=1=|\mu|\;\}.$$ 

For $P$ in $\H$, let $\m^+(4)_P$ be the stabilizer subgroup of $\m^+(4)$ at $P$. When $P=\infty$ we have
$$ \m^+(4)_{\infty}=\s^+(4).$$ 
Consider the transformations:

(i) the stretches $Z \mapsto  rZ$, $r>0$, $r \neq 1$,

(ii) the translations $Z \mapsto Z+a$,

(iii) the transformations $Z \mapsto \lambda Z$, $|\lambda|=1$, and  

(iv) $Z \mapsto \lambda Z \bar \lambda$, $|\lambda|=1$.

\noindent It is easy to see that the stretches can be expressed as a product of two inversions and the transformations of the type (ii), (iii) and (iv) can be expressed as a product of two reflections in hyperplanes.  Since $\s^+(4)$ is generated
by the transformations describe above, every element
in $\s^+(4)$ can be expressed as a product of inversions. Hence $\s^+(4)
\subset \m(4)$.

\noindent Note that the translations are
transitive on $\H$ and the map $g: Z \mapsto {\bar Z}^{-1}$ carries
$\infty$ to $0$. This shows that the group $\m(4)$ acts transitively on $\hat \H$. Hence for all $P \in \hat \H$, there exists $f \in \m(4)$
such that $f(\infty)=P$, and moreover,  $\m^+(4)_P=f\;o\;\s^+(4)\;o\;f^{-1}$. Thus $\m(4)$ contains $\cup_{P \in \hat \H}
\m^+(4)_P $. Since every element of $\m^+(4)$ has a fixed point on $\hat
\H$, we have $\m^+(4)=\cup_{P \in \hat \H}\m^+(4)_P$. Hence $\m^+(4) \subset \m(4)$.
 Since the group $G$ is generated by $\m^+(4)$ 
and the map $g:Z \mapsto \bar Z$ of $\m(4)$, hence 
$G \subset \m(4)$. 

This completes the proof.
\end{proof}

Note that the group $\m(4)$ acts as the full group of isometries of $\h^5$. In particular we have
\begin{corollary}
The group $\G$ acts as the group of orientation-preserving isometries of $\h^5$. In fact, the group of orientation-preserving isometries of $\h^5$ is isomorphic to $P\G$. 
\end{corollary}
\begin{proof}
Recall that, $P\G=\G/\{\lambda I: \lambda \in \R^{\ast}\}$. Hence the result follows.  
\end{proof}

\section {Conjugacy classes in $\G$}\label{cc}
\begin{theorem}
The conjugacy classes in $\G$ are represented by 

$(i)$ $T_{r, \theta}=
\begin{pmatrix}re^{i\theta} & 1 \\ 0 & re^{i
      \theta} \end{pmatrix}$, $0 \leq \theta \leq \pi$, $r>0$.

$(ii)$ $D_{r, \theta,\phi}=
\begin{pmatrix} re^{i\theta} & 0 \\ 0 & re^{i
 \phi} \end{pmatrix}$, $0 \leq \theta, \phi \leq \pi$,  $r>0$, $\theta \neq \phi$.

$(iii)$ $D_{r,s,\theta,\phi}=
\begin{pmatrix} re^{i\theta} & 0 \\ 0 & se^{i
 \phi} \end{pmatrix}$, 
$0 \leq \theta,\phi \leq \pi$,  $r>0$, $s>0$, $r \neq s$, $\theta \neq \phi$.

\end{theorem}
\begin{proof}
Consider $\H$ as the
two-dimensional left vector-space over $\C$
with basis $\{1, j\}$, i.e. we take $\H=\C \oplus \C j$. We identify $\hat \H$ with the boundary sphere of $\h^5$. 

Let $A=\begin{pmatrix}a & b \\ c & d \end{pmatrix}$ be an element in $\G$. The induced M$\o$bius transformation on $\hat \H$ is 
$f: Z \mapsto (aZ+b)(cZ+d)^{-1}$. We have seen that every element of
$\m^+(4)$ has a fixed point on $\hat \H$. 
Conjugating $f$ we can take the
fixed point to be $\infty$. So, upto conjugation, we take
$A={\begin{pmatrix} a  & b \\
0 & d \end{pmatrix}}$.  
 Since every quaternion is conjugate to an element in $\C$, 
let $a=v\{le^{i \theta}\}v^{-1}$, $d=w\{me^{i \phi}\}w^{-1}$, 
where $l$, $m$ non-zero reals, $0 \leq {\theta, \phi} \leq 2 \pi$.
Let $D=\begin{pmatrix} v & 0 \\ 0 & w \end{pmatrix}$. 
Then $DAD^{-1}= \begin{pmatrix} le^{i \theta} & v^{-1}bw \\
0 & me^{i \phi}\end{pmatrix}$. 
 Now note that $e^{i \theta}$ and $e^{-i
      \theta}=e^{i(2 \pi-\theta)}$ are conjugate to each other in $\H$.
If one, resp. both of $\theta$, $\phi$ are greater than $\pi$,  
 conjugate $DAD^{-1}$ by one, resp. both of the matrices
$\begin{pmatrix} j & 0 \\ 0 &
1 \end{pmatrix}$, 
$\begin{pmatrix} 1 & 0 \\ 0 &
j \end{pmatrix}$.   
 
Thus every element in $\G$ is conjugate to an element of the form
 $U=\begin{pmatrix} re^{i \theta} &
  d \\ 0 & se^{i \phi} \end{pmatrix}$, $d \in \H$, $0 \leq \theta, \phi
\leq \pi$, $r$, $s$ are positive reals.  

Now consider a matrix $U$ as above. 

 \noindent {\it Case (i). }Suppose $r=s \;,
\theta \neq \phi$. Then the induced M$\o$bius transformation by $U$ is  
$$f: Z \mapsto e^{i \theta} Z e^{-i \phi}+b,\;\;b=s^{-1}de^{-i\phi}$$
We write $Z=z+wj$,
$b=b_0+b_1j$. Then we have,
$$f(z+wj)=(e^{i(\theta-\phi)}z+b_0)+(e^{i(\theta+\phi)}w+b_1)j.$$
Let $Z_0=z_o+z_1j$, where 
$z_0= b_0(1-e^{i(\theta-\phi)})^{-1}$ and 
$z_1= b_1(1-e^{i(\theta+\phi)})^{-1}$.
Conjugating $f$ by the map $g:Z \mapsto Z-Z_0=X=x_0+x_1j$ we have 
\begin{eqnarray*}
& &gfg^{-1}(X)= gf(X+Z_0)\\
&=&g\big((e^{i(\theta-\phi)}(x_0+z_0)+b_0)+
(e^{i(\theta+\phi)}(x_1+z_1)+b_1)j\big)\\
&=&(e^{i(\theta-\phi)}x_0+c_0)+
 (e^{i(\theta+\phi)}x_1+c_1)j\:,
%\;\hbox{ where,}\;\\
%& & \hspace{2cm} c_0=\{e^{i(\theta-\phi)}z_0+b_0-z_0\}=0,\\
%& & \hspace{2 cm}c_1=\{e^{i(\theta-\phi)}z_1+b_1-z_1\})=0. \\
\end{eqnarray*}
where,  $c_0=\{e^{i(\theta-\phi)}z_0+b_0-z_0\}=0,\;
c_1=\{e^{i(\theta+\phi)}z_1+b_1-z_1\})=0$.
 
\noindent That is, $gfg^{-1}(X)=e^{i \theta}Xe^{-i \phi}$.
Hence lifting $g$ in $\G$ we see that the matrix $U$ is conjugate to the
diagonal matrix
$D_{r,\theta, \phi}=
\begin{pmatrix} re^{i \theta} & 0 \\ 0 & re^{i \phi} \end{pmatrix}$, $0
\leq \theta, \phi \leq \pi.$ 

\noindent{\it Case (ii). }Let $r \neq s$. That is, $U$
induces the transformation
$$\alpha: Z \mapsto ce^{i \theta} Z e^{-i \phi}+b,\;\;
c=rs^{-1} \neq 1, \; c>1.$$
It is easy to see that 
$$\beta \alpha \beta^{-1}=\alpha_0,$$ 
where 
$$\alpha_0: W \mapsto  ce^{i \theta} W e^{-i \phi},\;\;\beta: Z \mapsto Z-Z_0=W,$$
$$Z_0={b_0(1-ce^{i(\theta-\phi)})^{-1}}+{b_1(1-ce^{i(\theta+\phi)})^{-1}}j.$$
Hence the matrix $U$ is conjugate to 
 $$D_{r,s,\theta,\phi}=
\begin{pmatrix}re^{i\theta} & 0 \\ 0 & se^{i\phi} \end{pmatrix},$$
$0 \leq
      \theta, \phi \leq \pi$, $r$, $s$ are positive reals, $r \neq s$. 

\noindent {\it Case (iii). }Let $r=s$, $\theta=\phi$. In
this case $U$ acts as 
$\tau:Z \mapsto e^{i \theta}Ze^{-i\theta}+b$. We see that
$$\eta\tau\eta^{-1}=\tau_1,$$
 where
$$\tau_1: W \mapsto  e^{i\theta} W e^{-i\theta}+b_0,$$ 
$$\eta: Z \mapsto Z-Z_0, \;\;Z_0=b_1(1-e^{i 2 \theta})^{-1}j.$$  
 If $b_0 \neq 0$, conjugating $\tau_1$ by the map  $Z \mapsto b_0^{-1}Z$, 
we get $\tau_{\theta}:W \mapsto e^{i \theta}W e^{-i\theta}+1$.
 Otherwise, $\tau_1$ is conjugate
to $\tau_{0, \theta}:W \mapsto e^{i \theta}W e^{-i\theta}$. 
 Thus $\tau$ is conjugate to either
$\tau_{\theta}$ or $\tau_{0, \theta}$. This implies that any matrix 
$\begin{pmatrix} re^{i\theta} & d \\
 0 & re^{i \theta} \end{pmatrix} $,  $b_0 \neq 0$, in $\G$ is
 conjugate to $T_{r, \theta}=\begin{pmatrix} re^{i\theta} & 1 \\
0 & re^{i \theta} \end{pmatrix}$,
 $0 \leq \theta \leq \pi$. 

Now note that the matrices $D_{r,s,\phi,\psi}$ or $D_{r, \phi, \psi}$ have at least one fixed point on $\H$, whence
$T_{r, \theta}$ has no fixed point on $\H$. Hence $T_{r,\theta}$ is not conjugate to any of the
matrices $D_{r,s,\phi,\psi}$ or $D_{r, \phi, \psi}$. Also if $(r, \phi) \neq (s, \psi)$,  $T_{r, \theta}$ is not conjugate to $T_{s, \phi}$. 
Moreover the embedded images of $D_{r,s, \theta, \phi}$, $D_{r, \theta, \phi}$ and $D_{r, \theta, \theta}$ in $\mathbb GL(4, \C)$ have distinct characteristic polynomials. Hence no two of $D_{r,s, \theta, \phi}$, $D_{r, \theta, \phi}$ and $D_{r, \theta, \theta}$ are  conjugate to each-other. 

This completes the proof.
\end{proof}

\begin{corollary} 
Let $A$ be an element in $\G$ and $A_{\C}$ in be the 
corresponding element in  $\g(4, \C)$. Then all the
co-efficients of the characteristic polynomial of $A_{\C}$ are
real. In particular, determinant of $A_{\C}$ is a real positive
number.
\end{corollary}

\begin{proof}
 Since the characteristic polynomial
is a conjugacy invariant for any matrix, the corollary follows by
embedding the conjugacy class representatives of $\G$ into
$\g(4, \C)$. 
\end{proof}

\begin{remark}
From the action of the conjugacy class representatives on $\hat \H$ it is clear that an element $A$ in $\G$ acts as a (i) $2$-rotatory (resp. $1$-rotatory) elliptic if it is conjugate to $D_{r, \theta, \phi}$, $\theta \neq \phi$ (resp. $\theta = \phi \neq 0,\pi$), (ii) $A$ acts as $2$-rotatory hyperbolic, $1$-rotatory hyperbolic or a stretch if it is conjugate to $D_{r,s, \theta, \phi}$, and $\theta \neq \phi$, $\theta=\phi \neq 0$, or $\theta=\phi=0,\pi$ respectively, (iii) $A$ acts as a $1$-rotatory parabolic or a translation if it is conjugate to $T_{r, \theta}$, and $\theta \neq 0$ or $\theta=0,\pi$ respectively. 
\end{remark}

\section{Centralizers, and the $z$-classes in $\G$}\label{z-class}
We now compute the centralizers of each conjugacy class representative. First consider,
$T=\begin{pmatrix}re^{i \theta} & 1 \\ 0 & re^{i \theta}\end{pmatrix}$. 
Let $A=\begin{pmatrix} a & b \\ c & d \end{pmatrix}$ be an element in $Z(T)$. Then  the relation $AT=TA$ yields the following four equations:
$$ (i) \;r ae^{i \theta}=re^{i \theta}a +c,\;\; (ii) \;rce^{i \theta}= re^{i \theta}c,$$
$$(iii)\; a + rb e^{ i \theta}= d+ r e^{i \theta} b,\;\;(iv)\; c+rde^{i \theta}=re^{i\theta}d.$$
From $(ii)$ it follows at once that $c \in Z(re^{i\theta})$. If $c \neq 0$, then (i) yields, 
$$1=rc^{-1}ae^{i \theta}-rc^{-1}e^{i \theta}a=rc^{-1}ae^{i \theta}-re^{i \theta}c^{-1}a.$$ 
Equating the real parts on both sides we get $1=0$, which is impossible. Hence we must have $c=0$. Thus, $a$, $d$ are also in $Z(re^{i \theta})$. Now $(iii)$ yields,
$$rbe^{i \theta}-re^{i \theta}b=d-a.$$
Note that $d-a$ is in $Z(re^{ i \theta})$. Hence we must have $d=a$. This follows by using similar arguments as above. Thus
$$Z(T)=\bigg\{\begin{pmatrix}a & b \\ 0 & a\end{pmatrix}\;|\; a \in Z(re^{i \theta}),\;b  \in Z(re^{i \theta})\cup\{0\} \bigg \}.$$
Now note that 
$$Z(re^{i\theta})=\left \{\begin{array}{ll} \H^{\ast} \;\hbox { if $\theta=0$,}\\
\C^{\ast} \;\hbox { otherwise}
\end{array} \right.$$
and for $a \neq 0$, we have 
$$\begin{pmatrix} a & b \\0 & a \end{pmatrix}=\begin{pmatrix}a & 0 \\ 0 & a \end{pmatrix} {\begin{pmatrix} 1 & a^{-1} b \\0 & 1 \end{pmatrix}}.$$
This shows that 
$$Z(T)=\left\{\begin{array}{ll} \H^{\ast} \rtimes \T(2, \H) \;\hbox { if }\theta=0,\\
\C^{\ast} \rtimes \T(2, \C)  \;\hbox { otherwise,}
\end{array} \right.$$
where $\rtimes$ denotes the semi-direct product. 

\medskip Now consider $D=\begin{pmatrix} re^{i \theta} & 0 \\ 0 & se^{i \phi}\end{pmatrix}$. From the relation $AD=DA$ it follows that  
$$Z(D)=\bigg\{\begin{pmatrix}a & 0 \\ 0 & d \end{pmatrix} \;|\; 
a \in Z(re^{i \theta}),\; d \in Z(se^{i \phi}) \bigg\}.$$

{\it Case} (i).   $\;r \neq s$, or $\theta \neq \phi$. In this case,
$$Z(D)=\left\{\begin{array}{lll} \H^{\ast} \oplus \H^{\ast}\; \hbox{ if }\theta=0=\phi, r \neq s \\ \H^{\ast} \oplus \C^{\ast} \;\hbox{ if one of }\theta,\; \phi\;\hbox{ is non-zero}, \\ \C^{\ast} \oplus \C^{\ast} \;\hbox{ if }\theta \neq 0 \neq \phi \end{array} \right.$$

{\it Case} (ii).   $\;r=s$, $\theta=\phi$. We have
$$Z(D)=\left \{\begin{array}{ll} \G \;\hbox{ if } \theta=0,\\
{\mathbb G}L(2, \C)\;\hbox{ otherwise }. \end{array} \right.$$

From the description of the centralizers of the conjugacy class representatives, it is clear that there are exactly seven $z$-classes in $\G$. We list the $z$-class representatives in the following.

$$(i) \begin{pmatrix}1 & 0\\0 & 1\end{pmatrix}, \;\;\; (ii) \begin{pmatrix} r & 0 \\ 0 & s \end{pmatrix},\; r \neq s,\;r,s>0\;\;\; (iii) \begin{pmatrix} 1 & 1\\0 & 1\end{pmatrix},$$ 

$$(iv) \begin{pmatrix} \lambda & 1 \\ 0 & \lambda \end{pmatrix}, \;|\lambda|=1,\;\;\; 
(v) \begin{pmatrix} \lambda & 0 \\ 0 & \mu \end{pmatrix}, \;\lambda \neq \mu,\;\;\;
(vi) \begin{pmatrix} \lambda & 0 \\ 0 & 1 \end{pmatrix},\;\;\; 
(vii) \begin{pmatrix} \lambda & 0 \\ 0 & \lambda \end{pmatrix},$$
where $\lambda$, $\mu$ are elements in $\C-\R$.

This completes the proof  of \thmref{zclass}.

\section{Proof of \thmref{mainth}}\label{pmainth}
From the embedded images of the conjugacy class representatives it follows that $a_0>0$. 
Observe that $c_1$, $c_2$,
$c_3$ are conjugacy invariants in $\G$, as well as in $P\G$. Hence we can 
choose them as the conjugacy invariants for an isometry of $\h^5$. Now there are several cases. 

\noindent {\it Case (i).} $f$ is elliptic. Then $A$ is
  conjugate to a matrix $D_{r, \theta, \phi}=\begin{pmatrix}re^{i \theta} &
  0 \\ 0 & re^{i \phi}\end{pmatrix}$,   $0 \leq \theta, \phi \leq \pi$. So, 
$$\chi(A_{\C})=(x^2-2 r\cos \theta \;x+r^2)(x^2-2 r \cos \phi \;x+r^2).$$
We have,
$c_1=(\cos \theta+\cos \phi)^2=c_3$, $c_2=2(1+2 \cos \theta \cos
\phi)$.

There are the following possibilities.

{\it $A$ acts as an $1$-rotatory elliptic}, then $\theta=\phi$. So 
\begin{eqnarray*}
c_2 &=& 4\cos^2 \theta+2 \\
&=& (c_1+2) \end{eqnarray*} 

\begin{equation}\label{1re}
\hbox{i.e. }c_2= c_1+2. \end{equation}
Also observe that in this case,
\begin{equation}c_1<4, \;c_2<6. \end{equation}

{\it A acts as a $2$-rotatory elliptic}, that is, $\theta \neq
\phi$. In this case  
\begin{eqnarray*}
c_1&=&(\cos \theta+\cos \phi)^2 \\
&=& (\cos \theta-\cos \phi)^2+4\cos \theta \cos \phi
\end{eqnarray*}
\begin{equation}\label{2ro}
\hbox{i.e. }\;\;c_1=(\cos \theta-\cos \phi)^2+c_2-2.\end{equation}
To have the equality \eqnref{1re} we must have the bracket
term zero in expression \eqnref{2ro}. This is possible only when $\theta=\phi$. Hence,  if $\theta \neq \phi$,  we must have $( \cos \theta-\cos \phi)^2>0$ and hence
$$c_2<c_1 + 2.$$ 
{\it Case (ii).} $f$ is parabolic. Then $A$ is
  conjugate to a matrix ${T_{\theta}}=\begin{pmatrix}re^{i \theta} &
  1 \\ 0 & re^{i \theta} \end{pmatrix}$, $0 \leq \theta \leq \pi$. In this
  case,
$$\chi(A_{\C})=(x^2-2 r\cos \theta \;x+r^2)^2,$$ 
$$\hbox{i.e., }\;\;a_3=2 r\cos \theta,\;a_1=2r^3 \cos \theta, \;
a_2=2r^2(2 \cos^2 \theta+1), \;a_0=r^4.$$ 
Thus  
$c_1=4 \cos^2 \theta=c_3$ and it follows immediately that 
\begin{equation}\label{1rp} 
c_2=c_1+2 
\end{equation}
Note that when $\theta \neq 0$,
\begin{equation}
c_1<4,\;c_2<6. \end{equation}
\noindent If $A$ is a translation, then 
$c_1=4$, $c_2=6$,  and we must have  $A  \neq rI$ for $r$ real. 
If $A=rI$,  then it acts as the identity map. 

\noindent{\it Case (iii). }$f$ is hyperbolic. In this case, $A$ is
  conjugate to a matrix $D_{r,s, \theta, \phi}=\begin{pmatrix}re^{i \theta} &
  0 \\ 0 & se^{i \phi} \end{pmatrix}$, $0 \leq \theta, \phi \leq \pi$,
  $r>0$, $s>0$,  $rs^{-1}\neq 1$, and 
$$\chi(A_{\C})=(x^2-2 r\cos \theta \;x+r^2)(x^2-2s \cos \phi
  \;x+s^2).$$ 
\noindent After expanding the right-hand expression we have
$\chi(A_{\C})=$
$$x^4-2(r \cos \theta+ s\cos \phi)x^3+(4rs \cos \theta \cos
  \phi+r^2+s^2)x^2-2(r^2 s\cos \phi+rs^2 \cos \theta)x+r^2s^2.$$
$$\hbox{i.e. }a_3=r \cos \theta+ s\cos \phi,$$ 
$$a_2=r^2+s^2+4rs \cos \theta \cos \phi,$$ 
$$a_1=r^2s \cos \phi+ rs^2\cos \theta,\;\;a_0=r^2s^2.$$
Thus,  
$$c_1=\frac{(r \cos \phi+s \cos \theta)^2}{rs},$$ 
$$c_2=\frac{(r^2+s^2+4rs \cos \theta \cos \phi)}{rs},$$ 
$$c_3=\frac{(r \cos \theta+s\cos \phi)^2}{rs}.$$
We have, 
 \begin{eqnarray*}
c_1=c_3
&\Leftrightarrow&   r\cos \theta+s\cos \phi=r \cos \phi+ s \cos
\theta\\
&\Leftrightarrow& (r-s)(\cos \theta-\cos \phi)=0\\
&\Leftrightarrow& \theta=\phi, \hspace{.2cm} \hbox{since } \;r \neq s.
\end{eqnarray*} 
This shows that if the isometry induced by $A$ is hyperbolic, then 
$c_1=c_3$ if and only if $A$ is
  either a  stretch or an $1$-rotatory hyperbolic.

 Now let $A$ acts as an $1$-rotatory hyperbolic, i.e.
$\theta=\phi$. Hence,
$$c_1={\frac{(r+s)^2}{rs}}\cos^2 \theta=c_3,\;\;
c_2=\frac{(r^2+s^2+4rs \cos^2 \theta)}{rs}\:.$$ 
 
Observe that  
\begin{eqnarray*}
& & c_2=c_1+2 \\
&\Leftrightarrow& \frac{r^2+s^2+4rs \cos^2
    \theta}{rs}={\frac{(r+s)^2}{rs}}\cos^2 \theta +2\\
&\Leftrightarrow & r^2(1-\cos^2\theta)+s^2(1-\cos^2 \theta)-2rs(1-\cos^2
  \theta)=0\\
& \Leftrightarrow& (r-s)^2 \sin^2 \theta=0\\
& \Leftrightarrow& \theta=0,\pi, \hspace{.2cm} \hbox{since, }\;r \neq s,
\end{eqnarray*} 
For $\theta=0,\; \pi$, $A$ induces a stretch. 

When $\theta=\phi \neq 0, \pi$, note that 
$$c_2=c_1+2+\frac{(r-s)^2}{rs} \sin^2 \theta.$$
Since $\frac{(r-s)^2}{rs} \sin^2 \theta>0$, 
hence, when $\theta \neq 0, \pi$, we must have 
$$c_2>c_1 +2.$$
If $A$ induces a stretch, then  
 $c_1=\frac{(r+s)^2}{rs}$. 
We claim that in this case, $c_1 \neq 4$. If possible suppose $c_1=4$. Then 
\begin{eqnarray*}
c_1^2=16 &\Leftrightarrow &\frac{(r+s)^4}{r^2s^2} =16 \\
&\Leftrightarrow & (r+s)^4=16r^2s^2\\
& \Leftrightarrow&  \{(r+s)^2-4rs\}\{(r+s)^2+4rs\}=0\\
& \Leftrightarrow& (r-s)^2\{(r+s)^2+4rs\}=0\\
& \Leftrightarrow& r=s,\;\hbox{since,}\;(r+s)^2+4rs>0.
\end{eqnarray*}
Which is a contradiction to the fact that $A$ induces a stretch. Hence we must have $c_1 \neq 4$ in this case. Further we have
$$c_1-4\Leftrightarrow\frac{(r+s)^2}{rs}-4=\frac{(r-s)^2}{rs}.$$
Since $\frac{(r-s)^2}{rs}>0$, hence $c_1>4$.

Finally note that from \eqnref{1re} and \eqnref{1rp} it follows that 
 $A$ acts as an $1$-rotatory elliptic or an $1$-rotatory parabolic if and only if 
$$(i)\;c_2=c_1 +2, \hbox{ and }\;(ii)\;c_1<4.$$
In this case the isometries are classified as follows. 

If $A$ is an $1$-rotatory elliptic or an $1$-rotatory parabolic, then for $r>0$ and $\theta \neq 0$, $A$ is conjugate to $D_{r, \theta,\theta}$ or $T_{r, \theta}$ respectively.  
Hence, in  both cases, we have, 
$$\chi(A_{\C})=(x^2-2r \cos \theta\;x+r^2)^2.$$ 
 If $A$ acts as an $1$-rotatory elliptic, then the minimal
polynomial of $A_{\C}$ is given by 
$$m_1(A_{\C})=x^2-2 r\cos \theta\;x+r^2=x^2-(\det A_{\C}\;c_1^2)^{\frac{1}{4}}
\;x+\sqrt{\det A_{\C}}.$$  
But when $A$ acts as an $1$-rotatory parabolic, 
$m_1$ does not annihilate $A_{\C}$. So $m_1$ can not be the minimal
  polynomial of $A_{\C}$. In this case the minimal polynomial is equal to 
the characteristic polynomial and  is given by
$$m_2(A_{\C})=(x^2-2r \cos \theta \;x+r^2)^2.$$ 

This completes the proof of \thmref{mainth}.

\begin{acknowledgement}
I thank Ravi Kulkarni for suggesting the problem and for his comments on a first draft of this paper. Thanks are also due to John Parker and Ian Short for showing interest in this work, and for their comments on it. It is Ian Short who made me aware of the papers \cite{cao}, \cite{ps}.  
\end{acknowledgement}
\begin{acknowledgement}
Finally it is a great pleasure to thank the referee for his careful reading of this paper and many useful comments. 
\end{acknowledgement}

\end{document}